\documentclass[11pt]{amsart}
\usepackage{amsmath, amssymb, amscd, mathrsfs, url, pinlabel,hyperref, verbatim}
\usepackage[margin=1.25in,marginparwidth=1in,centering,letterpaper,dvips]{geometry}
\usepackage{color,dcpic,latexsym,graphicx,epstopdf,comment}
\usepackage[all]{xy}
\usepackage[dvipsnames]{xcolor}
\usepackage{tikz,pgfplots}
\usepackage{enumitem,upquote,tabularx,textcomp}
\usepackage[all]{hypcap}
\usepackage[color=blue!20!white,textsize=tiny]{todonotes}

\title{Instanton L-spaces and splicing}

\author{John A. Baldwin}
\email{john.baldwin@bc.edu}
\address{Department of Mathematics\\Boston College}

\author{Steven Sivek}
\email{s.sivek@imperial.ac.uk}
\address{Department of Mathematics\\Imperial College London}

\thanks{JAB was supported by  NSF FRG Grant DMS-1952707 and NSF CAREER Grant DMS-1454865.}

\makeatletter
\newtheorem*{rep@theorem}{\rep@title}
\newcommand{\newreptheorem}[2]{%
\newenvironment{rep#1}[1]{%
 \def\rep@title{#2 \ref{##1}}%
 \begin{rep@theorem}}%
 {\end{rep@theorem}}}
\makeatother

\newtheorem {theorem}{Theorem}
\newreptheorem{theorem}{Theorem}
\newtheorem {lemma}[theorem]{Lemma}
\newtheorem {proposition}[theorem]{Proposition}
\newtheorem {corollary}[theorem]{Corollary}
\newtheorem {conjecture}[theorem]{Conjecture}

\numberwithin{equation}{section}
\numberwithin{theorem}{section}

\theoremstyle{definition}

\newtheorem{remark}[theorem]{Remark}
\newtheorem*{remark*}{Remark}

\newlist{pcases}{enumerate}{1}
\setlist[pcases]{
  label=\bf{Case~\arabic*:}\protect\thiscase.~,
  ref=\arabic*,
  align=left,
  labelsep=0pt,
  leftmargin=0pt,
  labelwidth=0pt,
  parsep=0pt
}
\newcommand{\case}[1][]{%
  \if\relax\detokenize{#1}\relax
    \def\thiscase{}%
  \else
    \def\thiscase{~#1}%
  \fi
  \item
}

\newcommand{\Z}{\mathbb{Z}}

\newcommand{\C}{\mathbb{C}}

\newcommand{\F}{\mathbb{F}}
\newcommand{\Q}{\mathbb{Q}}

\newcommand\hfk{\mathit{HFK}}
\newcommand\hfkhat{\widehat{\hfk}}

\newcommand\SU{\mathit{SU}}
\newcommand\SL{\mathit{SL}}

\newcommand\SHI{\mathit{SHI}}

\newcommand\KHI{\mathit{KHI}}

\newcommand\Is{I^\#}

\DeclareFontFamily{U}{mathx}{\hyphenchar\font45}
\DeclareFontShape{U}{mathx}{m}{n}{
      <5> <6> <7> <8> <9> <10>
      <10.95> <12> <14.4> <17.28> <20.74> <24.88>
      mathx10
      }{}
\DeclareSymbolFont{mathx}{U}{mathx}{m}{n}
\DeclareFontSubstitution{U}{mathx}{m}{n}
\DeclareMathAccent{\widecheck}{0}{mathx}{"71}

\newcommand{\hfhat}{\widehat{\mathit{HF}}}

\newcommand{\id}{\operatorname{id}}

\newcommand{\coker}{\operatorname{coker}}
\newcommand{\pt}{\mathrm{pt}}

\newcommand{\im}{\operatorname{Im}}

\usetikzlibrary{calc,intersections}
\tikzset{every picture/.style=thick}
\tikzset{link/.style = { white, double = black, line width = 1.75pt, double distance = 1.25pt, looseness=1.75 }}
\tikzset{crossing/.style = {draw, circle, dotted, minimum size=0.5cm, inner sep=0, outer sep=0}}
\pgfplotsset{compat=1.12}

\begin{document}

\begin{abstract}
We prove that the 3-manifold obtained by gluing the complements of  two nontrivial knots in   homology 3-sphere instanton $L$-spaces, by a map which  identifies meridians with Seifert longitudes,  cannot be an instanton $L$-space. This recovers the recent theorem of Lidman--Pinz{\'o}n-Caicedo--Zentner that the fundamental group of every closed, oriented, toroidal 3-manifold admits a nontrivial  $\mathit{SU}(2)$-representation, and consequently Zentner's earlier  result that the fundamental group of every closed, oriented $3$-manifold besides the 3-sphere admits a nontrivial $\mathit{SL}(2,\C)$-representation. \end{abstract}

\maketitle

\section{Introduction}\label{sec:intro}
A  productive way to study 3-manifolds is to study homomorphisms from their fundamental groups to simple non-abelian groups like $SU(2)$. This strategy was used by Kronheimer and Mrowka in their celebrated proof of the Property P Conjecture  \cite{km-p}. More precisely, they proved that if $K\subset S^3$ is a nontrivial knot, then there is a nontrivial homomorphism \[\pi_1(S^3_1(K))\to SU(2),\] certifying that the surgered manifold is not a homotopy sphere. It is natural to  conjecture more generally  that \emph{every} homology 3-sphere $Y\not\cong S^3$ admits a nontrivial homomorphism \[\pi_1(Y)\to SU(2).\] This is a generalization of the Poincar{\'e} Conjecture, and is listed as \cite[Problem 3.105(A)]{kirby-list} on Kirby's  famous problem list.  Lidman, Pinz{\'o}n-Caicedo, and Zentner recently proved this conjecture for toroidal homology 3-spheres \cite{lpcz}, recovering in the process Zentner's beautiful result \cite{zentner} that the conjecture holds with $SL(2,\C)$ in place of $SU(2)$. 

In this article, we take inspiration from  Heegaard Floer homology to prove a result in the  instanton Floer setting, which we then use to give different and simpler proofs of the main results of \cite{lpcz,zentner}, as described below. 

An \emph{instanton $L$-space} is a rational homology $3$-sphere $Y$ whose framed instanton homology \cite{km-yaft,scaduto} has the smallest dimension possible over $\Q$:
%\todo{SS: should specify coefficients somewhere}
\[\dim \Is(Y) = |H_1(Y;\Z)|.\] The topological significance of this notion is that if $Y$ is an \emph{integer} homology 3-sphere, then $Y$ is \emph{not} an instanton $L$-space if and only if the reduced instanton homology $\hat{I}(Y)$ is nonzero  \cite[Theorem~1.3]{scaduto}, in which case there exists a nontrivial homomorphism  \[\pi_1(Y)\to\SU(2).\] (See also \cite[Theorem 4.2]{bs-stein}.) Our main theorem is the following:%\footnote{Our result pertains more generally to rational homology 3-spheres and  irreducible $\SU(2)$-representations.} 

\begin{theorem}
\label{thm:main}
 Let $C$ be the splice of nontrivial knots $\alpha\subset A$ and $\beta\subset B$ in homology sphere instanton $L$-spaces, formed by gluing their   complements  by a map that identifies meridians with Seifert longitudes. Then \[\dim \Is(C) \geq 1+4\cdot\dim\KHI(A,\alpha,g(\alpha))\cdot\dim\KHI(B,\beta,g(\beta)) \geq 5.\]
%\todo{SS: made the actual bound (with $+1$) explicit}
In particular, the homology sphere $C$ is  not an instanton $L$-space.
\end{theorem}

Here, $\KHI$ refers to a version of Floer's \emph{instanton knot homology}, defined in \cite[\S7.6]{km-excision}. This theory was a key tool, for instance, in Kronheimer and Mrowka's proof that Khovanov homology detects the unknot \cite{km-unknot}, and is conjectured to be isomorphic to a version of Heegaard knot Floer homology (as mentioned below).

%\begin{remark}
%The instanton knot homology groups in Theorem \ref{thm:main} are 1-dimensional iff the corresponding knots are fibered. We thus have the  improved lower bounds: \[\dim\Is(C) > 
%\begin{cases}
%8,& \textrm{if only one of $\alpha$ and $\beta$ are fibered,}\\
%16,& \textrm{if neither knot is fibered.}
%\end{cases}
%\]
%\end{remark}

A Heegaard Floer analogue of Theorem \ref{thm:main}, with coefficients in $\Z/2\Z$ rather than $\Q$, was first established by Hedden and Levine in \cite{hedden-levine}, and later generalized by Eftekhary \cite{eaman,eaman2} and then by Hanselman, Rasmussen, and Watson \cite{hrw}. Their approaches all used bordered Heegaard Floer homology or something similar. By contrast, our proof of Theorem \ref{thm:main} is largely ``theory-independent," meaning that it can be readily adapted to give alternative proofs of the corresponding results in Heegaard Floer and monopole Floer homology as well.  Indeed, the only tools it requires to prove these results over a field $\F$ are a surgery exact triangle with $\F$ coefficients, a suitable excision theorem, and an adjunction inequality.  We remark that the Heegaard Floer version over a field of characteristic zero would also follow from Theorem~\ref{thm:main} together with the conjectural isomorphisms
\[ I^\#(Y;\Q) \cong \hfhat(Y;\Q) \quad\text{and}\quad \KHI(Y,K,g(K);\Q) \cong \hfkhat(Y,K,g(K);\Q), \]
which are special cases of \cite[Conjecture~7.24]{km-excision}.

The advantage of proving a result like Theorem \ref{thm:main} in the instanton Floer setting has to do, of course, with the connection between framed instanton homology and the fundamental group, highlighted above. In particular, Theorem \ref{thm:main} has as a corollary the following result of Lidman, Pinz{\'o}n-Caicedo, and Zentner, alluded to at the beginning. Their proof used a different version of  instanton Floer homology, together with arguments involving exact triangles, cabling, the pillowcase, and holonomy perturbations. 

%In the  following corollaries, $Y$ is a closed, oriented 3-manifold.

\begin{corollary}[\cite{lpcz}]
\label{cor:toroidal}
There exists a nontrivial homomorphism $\pi_1(Y)\to SU(2)$ for every closed, oriented, toroidal 3-manifold $Y$.
\end{corollary}

Corollary \ref{cor:toroidal}, in turn, recovers the result of Zentner below, whose original proof involved extensive analysis of the pillowcase, holonomy perturbations, and degree-1 maps.

\begin{corollary}[\cite{zentner}]\label{cor:SL2}
There exists a nontrivial homomorphism $\pi_1(Y)\to SL(2,\C)$ for every closed, oriented 3-manifold  $Y\not\cong S^3$.
\end{corollary}

We remark that  Hanselman, Rasmussen, and Watson proved \cite{hrw} something  more general than Theorem \ref{thm:main} in the Heegaard Floer setting. In particular, they showed that the Heegaard Floer homology of \emph{any} toroidal manifold (not necessarily coming from a splice of knots in homology sphere $L$-spaces) has dimension at least five (over $\Z/2\Z$). We conjecture that the analogous statement holds  for framed instanton homology:

\begin{conjecture}\label{conj:toroidal}
If $Y$ is a closed, oriented, toroidal rational homology 3-sphere, then $\dim\Is(Y) \geq 5$.
\end{conjecture}

Moreover,  toroidal manifolds which achieve the bound of five in Heegaard Floer homology are  classified in  \cite{hrw}---they are certain splices of trefoils. It would be  interesting  to try and prove that framed instanton homology is 5-dimensional as well for these trefoil splices.

If Conjecture \ref{conj:toroidal} holds, then a toroidal rational homology sphere $Y$ with $|H_1(Y)|< 5$ is not an instanton L-space. This would imply the following, by \cite[Corollary 4.10]{bs-stein}:

\begin{conjecture}
\label{conj:toroidalqhs}
There exists an irreducible homomorphism $\pi_1(Y)\to SU(2)$ for every toroidal rational homology 3-sphere $Y$ with $|H_1(Y)|<5$.
\end{conjecture} 

Combined with the classification of $SU(2)$-abelian Seifert fibered 3-manifolds in  \cite[Theorem 1.1]{sz-menagerie}, and the fact that  fundamental groups of hyperbolic 3-manifolds always admit irreducible $SL(2,\C)$-representations, a proof of Conjecture \ref{conj:toroidalqhs} would  be enough  to    classify rational homology 3-spheres $Y$ with $|H_1(Y)|< 5$ which admit irreducible homomorphisms \[\pi_1(Y)\to SL(2,\C),\]  extending Zentner's result for homology 3-spheres  stated in Corollary \ref{cor:SL2}.

We prove Theorem \ref{thm:main} and its two  corollaries in \S\ref{sec:proof}, after a review of framed instanton homology in \S\ref{sec:background}. The proof of Theorem \ref{thm:main} relies on Theorem \ref{thm:nonzero}, which we prove in  \S\ref{sec:appendix} and  may be of independent interest.

We thank Josh Greene, Matt Hedden, Adam Levine, and Tye Lidman for helpful conversations in the course of this work, and to the anonymous referees, whose feedback helped improve this paper. Thanks also go to Tye Lidman, Juanita Pinz{\'o}n-Caicedo, and Raphael Zentner for their lovely paper \cite{lpcz}, which inspired this project.

\section{Background} \label{sec:background} The following is a  brief review of   framed instanton homology, limited to the background needed to follow the flow of our arguments.  Instanton  homology can be defined over $\Z$, but  we will work with coefficients in  $\Q$ throughout. For more details, see \cite[\S7]{km-excision}. 

Let $Y$ be a closed, oriented 3-manifold, and $\lambda\subset Y$ a multicurve in $Y$ which is \emph{admissible}, meaning that it intersects some closed surface of positive genus in an odd number of points. One can associate to this data an \emph{admissible} $\mathit{SO}(3)$-bundle over $Y$ with second Stiefel--Whitney class Poincar\'e dual to $\lambda$,
and then define a corresponding instanton Floer homology group $I_*(Y)_\lambda$. 

Suppose $\Sigma\subset Y$ is a closed, oriented surface of  genus $g>0$, and $\pt\in Y$ is a point. These submanifolds induce commuting linear operators $\mu(\Sigma),\mu(\pt)$ on $I_*(Y)_\lambda$ whose simultaneous eigenvalues are contained in the set 
\[\{(i^r\cdot 2k, (-1)^r\cdot 2)\},\] for $0\leq r \leq 3$ and $0\leq k\leq g{-}1$. We let \[I_*(Y|\Sigma)_\lambda\subset I_*(Y)_\lambda\] denote the (generalized) simultaneous $(2g{-}2, 2)$-eigenspace of these operators. These groups depend only on the homology classes of $\Sigma$ and $\lambda$ in $H_2(Y;\Z)$ and $ H_1(Y;\Z/2\Z)$, respectively. %Moreover, $\mu(-\Sigma) = -\mu(\Sigma)$

Now let $Y$ be any closed, oriented $3$-manifold, and $\lambda\subset Y$ \emph{any} multicurve. The framed instanton homology of $(Y,\lambda)$ is defined by \[\Is(Y,\lambda):=I_*(Y\#T^3)_{\lambda \cup \lambda_T},\] where $\lambda_T$ is a circle fiber in $T^3$ (note that $\lambda \cup \lambda_T$ is automatically admissible in $Y\# T^3$). We will use the notation $\Is(Y)$ or $\Is(Y,0)$ for $\Is(Y,\lambda)$ when $\lambda$ is  nullhomologous mod 2.
Given a closed, oriented surface $\Sigma\subset Y$ of genus $g>0$, we  define \[\Is(Y,\lambda|\Sigma) := I_*(Y\#T^3|\Sigma)_{\lambda \cup \lambda_T}.\] By \cite[Lemma 2.5]{bs-trefoil}, we have the symmetry \begin{equation}\label{eqn:symm}\Is(Y,\lambda|\Sigma) \cong \Is(Y,\lambda|{-}\Sigma) \end{equation} relating the generalized $(2g{-}2)$ and $(2{-}2g)$-eigenspaces of $\mu(\Sigma)$ acting on $\Is(Y,\lambda)$.\footnote{Indeed, $\Is(Y,\lambda|{-}\Sigma)$ is the $(2{-}2g,2)$-eigenspace of $\mu(\Sigma),\mu(\pt)$ since $\mu(-\Sigma) = -\mu(\Sigma)$.} When $g\geq 2$, these \emph{extremal} eigenspaces are disjoint, and we define \[\Is(Y,\lambda)_\Sigma: = \Is(Y,\lambda|\Sigma)\oplus \Is(Y,\lambda|{-}\Sigma).\]

Floer's surgery exact triangle \cite{floer-surgery} associates to a framed knot $K \subset Y$ an exact triangle
\[ \cdots \to I_*(Y)_\lambda \to I_*(Y_0(K))_{\lambda\cup\mu} \to I_*(Y_1(K))_\lambda \to \cdots, \]
where $\mu$ is the image in $Y_0(K)$ of a meridian of $K$, provided that all three multicurves are admissible. As discussed in \cite[\S7.5]{scaduto}, this gives rise to a surgery exact triangle in framed instanton homology as well, which takes the form \[ \cdots \to I^\#(Y,\lambda) \to I^\#(Y_0(K),\lambda\cup \mu) \to I^\#(Y_1(K),\lambda) \to \cdots \] for any multicurves $\lambda$ and $\mu$. We can always arrange, via a judicious choice of $\lambda$, that these multicurves are nullhomologous mod 2 in their respective manifolds, as in \cite[\S2.2]{bs-concordance}, and so obtain an exact triangle of the form \[ \cdots \to I^\#(Y) \to I^\#(Y_0(K)) \to I^\#(Y_1(K)) \to \cdots \]
The maps in  these exact triangles are induced by the associated trace cobordisms, equipped with $\mathit{SO}(3)$-bundles whose second Stiefel-Whitney classes are Poincar{\'e} dual to certain surface cobordisms between the corresponding multicurves; see \cite[\S2.3]{abds}.

\section{The proofs}\label{sec:proof}
%Before proving Theorem \ref{thm:main}, we introduce some  notation and preliminary lemmas. %We will assume the reader is familiar with framed instanton homology and sutured instanton homology at the level of \cite{km-excision}.

%Let $Y$ be a closed, oriented $3$-manifold, and $\Sigma\subset Y$  a closed, oriented surface of genus at least two. For any $\lambda \in H_1(Y;\Z/2\Z)$,  the surface $\Sigma$ induces a linear map \cite[\S7]{km-excision}\[\mu(\Sigma):\Is(Y,\lambda)\to\Is(Y,\lambda),\] whose eigenvalues are even integers in the interval $[2{-}2g,2g{-}2],$ where $g=g(\Sigma)$. This map  depends only on the homology class $[\Sigma]\in H_2(Y;\Z),$ and satisfies \[\mu(-\Sigma) = -\mu(\Sigma).\]  Let $\Is(Y,\lambda|\Sigma)$ be the (generalized) $(2g{-}2)$-eigenspace of $\mu(\Sigma)$.   Then $\Is(Y,\lambda|{-}\Sigma)$ is  the $(2{-}2g)$-eigenspace of $\mu(\Sigma)$, and we have the following symmetry, by \cite[Lemma 2.5]{bs-trefoil}:   \begin{equation}\label{eqn:symm}\dim \Is(Y,\lambda|\Sigma) = \dim \Is(Y,\lambda|{-}\Sigma). \end{equation} Let us define \[\Is(Y,\lambda)_\Sigma: = \Is(Y,\lambda|\Sigma)\oplus \Is(Y,\lambda|{-}\Sigma).\]
Before proving Theorem \ref{thm:main}, we introduce some  notation and preliminary results. One of these is the theorem below, whose proof we defer to \S\ref{sec:appendix}.

%Now suppose that $\alpha$ and $\beta$ are  nontrivial knots in   homology spheres $A$ and $B$, respectively. Let $S$ be any splice  \[S = (A- N(\alpha)) \cup ({B- N(\beta)})\] obtained by gluing these  knot complements via an orientation-reversing homeomorphism of their boundaries which identifies  the Seifert longitudes. Let $\Sigma\subset S$ be the closed surface formed by gluing together minimal genus Seifert surfaces  for the two knots. %Let $\mu \subset S$ be a curve which intersects $\Sigma$ in one point, and  $\lambda \in H_1(S;\Z/2\Z)$  any mod 2 multiple of $\mu$. THINK ABOUT THIS. CAN I JUST TAKE ANY CLASS? 
%We   will use the proposition below in our proof of Theorem \ref{thm:main}, but postpone its proof until \S\ref{sec:appendix}.

\begin{theorem}\label{thm:nonzero}
Let $S$ be a splice of nontrivial knots $\alpha\subset A$ and $\beta\subset B$ in  homology spheres, formed by gluing their   complements  by a map that identifies  the Seifert longitudes. Then \[\dim \Is(S,\lambda)_\Sigma = 4\cdot\dim\KHI(A,\alpha,g(\alpha))\cdot\dim\KHI(B,\beta,g(\beta))\geq 4\]  for all $\lambda$, where $\Sigma\subset S$ is the union of minimal genus Seifert surfaces for $\alpha$ and $\beta$.
\end{theorem}

As remarked in the introduction, $\KHI(A,\alpha)$ (and similarly $\KHI(B,\beta)$) refers to the instanton knot homology studied by Kronheimer and Mrowka \cite[\S7.6]{km-excision}, defined as the sutured instanton homology of the complement $A\setminus N(\alpha)$ with sutures a pair of oppositely oriented meridians.  It admits an ``Alexander decomposition'' into summands indexed by integers $i$ with $|i| \leq g(\alpha)$; the top summand $\KHI(A,\alpha,g(\alpha))$ is nonzero, and it arises as the sutured instanton homology of the complement of a genus-$g(\alpha)$ Seifert surface.

The strategy behind our proof of Theorem \ref{thm:main} will involve relating the splice $C$ to a splice $S$ of the form described in the theorem above, via surgery. We will then use surgery exact triangles to show that the extremal subspace $\Is(S,\lambda)_\Sigma$ ``persists" and contributes at least four to the dimension of $\Is(C)$.

%\begin{proposition} With respect to the notation above, \label{prop:nonzero}  \[\dim \Is(S,\lambda)_\Sigma = 4\cdot \dim\KHI(A,\alpha,g(\alpha))\cdot\dim\KHI(B,\beta,g(\beta))\geq 4\]

%\end{proposition}

Suppose $\alpha\subset A$ is a knot in a homology sphere instanton $L$-space. Since $\dim \Is(A)=1$, the surgery exact triangle \[\cdots \to \Is(A) \xrightarrow{} \Is(A_{i}(\alpha)) \xrightarrow{} \Is(A_{i+1}(\alpha)) \to \cdots\] implies that \[\dim \Is(A_{i+1}(\alpha)) = \dim\Is(A_{i}(\alpha)) \pm 1\] for each integer $i$. In \cite[\S3]{bs-concordance} (see also \cite[Remark~6.3]{bs-concordance2}),
we proved that the sequence \[\{\dim \Is(A_i(\alpha))\}_{i\in\Z}\] either has one local minimum; or else it has two local minima at $i=\pm1$, satisfying \[\dim \Is(A_{-1}(\alpha))  = \dim \Is(A_1(\alpha)).\footnote{We  proved this for knots in the 3-sphere, but the only fact about $S^3$ we  used is  that $\dim \Is(S^3)=1$.  This is why we require $A$ and $B$ to be homology sphere instanton L-spaces: otherwise we cannot guarantee that the sequences $\{\dim I^\#(A_i(\alpha))\}_{i\in\Z}$ and $\{\dim I^\#(B_i(\beta))\}_{i\in\Z}$ behave as described.}\] In the first case, we define $\nu^\sharp(\alpha)$ to be the unique  $i$ at which the local minimum is achieved; in the second, we define $\nu^\sharp(\alpha)=0$. Note that $\nu^\sharp(\alpha)$ changes sign upon reversing the orientation of $A$. %Let $\mu\subset A$ be a meridian of $\alpha$, and   $\lambda \in H_1(A_0(\alpha);\Z/2\Z)$    a mod 2 multiple of $\mu$.
The following is a combination of \cite[Theorem 6.1]{bs-concordance2} and \cite[Proposition 3.3]{bs-concordance}:

\begin{proposition}
\label{prop:nu}
 If $\nu^\sharp(\alpha) = 0$ then for some $\lambda$,
\[\dim \Is(A_{0}(\alpha), \lambda) -1 = \dim \Is(A_{-1}(\alpha)) = \dim \Is(A_{1}(\alpha)).\] If $\nu^\sharp(\alpha) \neq 0$ then $\dim\Is(A_{0}(\alpha), \lambda)$ is independent of $\lambda$.\end{proposition}

The upshot of this discussion is the lemma below. It follows immediately from Proposition \ref{prop:nu}, and will be used shortly in the proof of Theorem \ref{thm:main}:

\begin{lemma} 
\label{lem:nu} Consider the surgery exact triangles: \begin{align*}
\cdots \to \Is(A) &\xrightarrow{F_{-1}} \Is(A_{-1}(\alpha)) \xrightarrow{G_0} \Is(A_0(\alpha), \lambda) \to \cdots\\
\cdots \to \Is(A)& \xrightarrow{F_0} \Is(A_{0}(\alpha),\lambda)  \xrightarrow{G_1} \Is(A_1(\alpha)) \to \cdots
\end{align*}
The following hold:
\begin{itemize}[noitemsep]
\item If $\nu^\sharp(\alpha)=0$, then for some $\lambda$, both $G_0$ and $F_0$ are injective.
\item If $\nu^\sharp(\alpha)>0$, then for any $\lambda$, both $F_0$ and $F_{-1}$ are injective.
\item If $\nu^\sharp(\alpha)<0$, then for any $\lambda$, both $G_0$ and $G_1$ are injective.
\qed
\end{itemize} 
\end{lemma}

\begin{proof}[Proof of Theorem \ref{thm:main}]
Let $\alpha\subset A$ and $\beta \subset B$ be as in the hypothesis of the theorem.  Then $C$ is the homology sphere formed by gluing the  knot complements \[C = (A- N(\alpha)) \cup ({B- N(\beta)}),\] via an orientation-reversing homeomorphism of their boundaries which identifies meridians with Seifert longitudes. Note that the theorem is insensitive to the overall orientation of $C$, so by possibly reversing the orientations of  $A$ and $B$ we can (and will) assume that either:
%\begin{itemize}
%\item if at least one of $\nu^\sharp(\alpha)$ and $\nu^\sharp(\beta)$ is zero, then the other is zero or negative;
%\item if both are nonzero, then $\cinvt(\alpha) < 0$.
%\end{itemize}
 \begin{align}
\label{case1}&\nu^\sharp(\alpha) \leq 0 \textrm{ and }\nu^\sharp(\beta) \leq 0, \textrm{ or}\\
\label{case2}&\nu^\sharp(\alpha) < 0\textrm{ and }  \nu^\sharp(\beta) > 0.
\end{align} We will henceforth use the shorthand \[A_i = A_i(\alpha)\textrm{ and }B_i = B_i(\beta)\] for surgeries on these knots. Let $\Sigma_\alpha$ and $\Sigma_\beta$ be minimal genus Seifert surfaces for $\alpha$ and $\beta$, and let $\hat\alpha$ and $\hat\beta$ be pushoffs of these knots in $\Sigma_\alpha$ and $\Sigma_\beta$.

Let $\gamma\subset C$ be a curve on the splicing torus of slope $-1$ with respect to the usual  coordinate system inherited from $A-N(\alpha)$. Note that  0-surgery on $\gamma$ with respect to its torus framing %(equivalently, $-1$-surgery with respect to its Seifert framing)
yields the connected sum
\begin{equation} \label{eqn:sum}
C_0(\gamma)\cong A_{-1}\#B_{-1}.
\end{equation}
Performing $1$-surgery on $\gamma$ amounts to changing the gluing map by a Dehn twist, and yields another splice
\[ S := C_1(\gamma) \]
of the form described in Theorem~\ref{thm:nonzero}, in which the Seifert longitudes of $\alpha$ and $\beta$ are identified. See Figure~\ref{fig:splicing-schematic} for a schematic of these surgeries. Let \[\Sigma:=\Sigma_\alpha \cup \Sigma_\beta \subset S\] be the closed surface of genus at least two obtained as the union of    the Seifert surfaces $\Sigma_\alpha$ and $\Sigma_\beta$ along their boundaries.

\begin{remark}
We have seen that $C$ is related  to both $A_{-1}\#B_{-1}$ and  $S$ via surgery. We have some understanding of the framed instanton homology of the latter two manifolds thanks to Lemma \ref{lem:nu} and Theorem \ref{thm:nonzero}. 
As mentioned earlier, our strategy in proving Theorem \ref{thm:main} will be to use surgery exact triangles involving these manifolds, and other surgeries in which  $\Sigma$ is compressed, to show that the extremal eigenspaces of the framed instanton homology of $S$ with respect to $\Sigma$ contribute at least four to the dimension of $\Is(C)$.
\end{remark}

\begin{figure}
\begin{tikzpicture}[scale=0.70]
\begin{scope} % exterior of \alpha
\path[fill=blue!5] (1.3,1.3/1.5) rectangle ++(-0.75,1.5);
\draw[blue,thin] (1.3,1.3/1.5) ++(-0.75,0) -- node[left,pos=0.5,inner sep=1pt] {\tiny$\hat\alpha$} ++(0,0.95) ++(0,0.1) -- ++(0,0.45);
\draw (0,0) -- ++(1.5,1) -- ++(0,1.5) ++(0,0) -- ++(-1.5,-1) -- ++(0,-1.5);
\draw (1.5,2.5) -- ++(-1.75,0) to[out=180,in=180,looseness=3] ++(-1.5,-1) -- ++(1.75,0);
\begin{scope}
\draw[thin,densely dotted] (0,1) -- ++(1.5,0);
\clip (0,-0.1) rectangle (-2.8,1.5);
\draw[thin,densely dashed] (1.5,1) -- ++(-1.75,0) to[out=180,in=180,looseness=3] ++(-1.5,-1) -- ++(1.75,0);
\clip (0,-0.1) rectangle (-2.8,0.25);
\draw (1.5,1) -- ++(-1.75,0) to[out=180,in=180,looseness=3] ++(-1.5,-1) -- ++(1.75,0);
\end{scope}
\draw (-2.7725,0.25) -- ++(0,1.5);
\draw[red,-latex] (1.5,1.2) -- node[above,pos=0.8,inner sep=5pt] {\small$\mu_\alpha$} ++(-1.5,-1);
\draw[blue,-latex] (1.3,1.3/1.5) -- node[left,pos=0.6,inner sep=1pt] {\small$\lambda_\alpha$} ++(0,1.5);
\end{scope}

\begin{scope} % T^2 x I
\draw[thin] (2.75,0) coordinate (g0) -- ++(0,1.5) -- ++(1.5,1);
\draw[thin,densely dotted] (g0) -- ++(1.5,1) -- ++(0,1.5);
\path (g0) -- coordinate[pos=0.33] (g4) coordinate[pos=0.36] (g5) coordinate[pos=0.58] (g1) coordinate[pos=0.62] (g2) ++(1.5,2.5) coordinate (g3);
\draw[thin,densely dashed] (1.75,0) coordinate (a) -- ++(1.5,1) coordinate (b) -- ++(0,1.5) coordinate (c);
\draw[thin,densely dashed] (b) -- ++(1.5,0) coordinate (d) -- ++(0,1.5);
\draw[thin,densely dashed] (d) -- ++(1.5,-1);
\begin{scope} % draw most of the curve parallel to the (-1)-framed one
%\path (2.75,0) -- ++(0,1.5) -- coordinate[pos=0.8] (p0) ++(1.5,1);
\path (c) -- coordinate[pos=0.1] (p0) coordinate[pos=0.9] (p4) ++(-1.5,-1);
\path (c) -- ++(1.5,0) -- coordinate[pos=0.1] (p1) ++(1.5,-1);
\path (p1) -- ++(0,-1) coordinate (p2) (p4) -- ++(0,-1.25) coordinate (p5);
\draw[orange,thin] (p0) -- ++(0,-0.63) ++(0,-0.07) -- ++(0,-0.15) ++(0,-0.10) -- ++(0,-0.05) coordinate (p6);
\draw[orange,thin] (p6) -- ($(p6)!0.175!(p2)$) ($(p6)!0.325!(p2)$) -- (p2) -- ++(0,-0.25) -- ++(1.2,-0.8) coordinate (p3);
\end{scope}
\begin{scope} % draw +1-framed push-off
\clip (g0) -- ++(1.5,1) -- ++(0,1.5) -- ++(-1.5,-1) -- ++(0,-1.5);
\path (g0) ++(0,0.2) coordinate (h0) -- coordinate[pos=0.25] (h2) coordinate[pos=0.4] (h3) ++(1.5,2.5) coordinate (h1);
\draw[purple,thin] (h0) -- ($(h0)!0.175!(h2)$) ($(h0)!0.3!(h2)$) -- (h2) (h3) -- ($(h3)!0.175!(h1)$) ($(h3)!0.225!(h1)$) -- (h1);
\draw[purple,thin] (h0) ++(0,-1.5) -- ++(1.5,2.5);
\draw[purple,thin] (h2) to[out=59,in=239] ($(h2)!0.5!(h3)+(-31:0.2)$) coordinate (h4);
\begin{scope} % wind the purple curve around the green one
\path (h4) ++(0.05,0.105) coordinate (cse) ++(-0.15,0.08) coordinate (cnw);
\draw[purple,thin] (h4) to[out=59,in=239] (h3);
\path[fill=white] (cse) rectangle (cnw);
\end{scope}
\end{scope}
\draw[Green,very thick] (g0) -- ($(g0)!0.375!(g4)$) ($(g0)!0.475!(g4)$) -- (g4) (g5) -- (g1) (g2) -- node[pos=0.3,right] {\small$\gamma$} (g3);
\begin{scope} % draw the rest of the curve parallel to the (-1)-framed one
\draw[orange,thin] (p3) -- ($(p3)!0.79!(p5)$) ($(p3)!0.805!(p5)$) -- (p5) -- ++(-0.15,-0.1) ++(1.5,1) -- ++(-0.15,-0.1) -- ++(0,-0.25);
\draw[orange,thin,densely dotted] (p6) -- ++(0,-0.3);
\end{scope}
\draw (c) -- ++(-1.5,-1) -- ++(0,-1.5);
\draw (a) -- ++ (4.5,0) -- ++(0,1.5) -- ++(-1.5,1) -- ++(-1.5,0);
\draw (a) ++(0,1.5) -- ++(4.5,0);
\end{scope}

\begin{scope}[xshift=8cm] % exterior of \beta
\path[fill=blue!5] (-1.5,1.2) -- ++(1,0) coordinate (beta1) -- ++(1.5,-1) coordinate (beta2) -- ++ (-1,0);
\draw[blue,thin] (beta1) -- ($(beta1)!0.3!(beta2)$) ($(beta1)!0.375!(beta2)$) -- node[above,pos=0.7] {\tiny$\hat\beta$} (beta2);
\draw (0,0) -- ++(-1.5,1) -- ++(0,1.5) ++(0,0) -- ++(1.5,-1) -- ++(0,-1.5);
\draw (-1.5,2.5) -- ++(1.75,0) to[out=0,in=0,looseness=3] ++(1.5,-1) -- ++(-1.75,0);
\begin{scope}
\draw[thin,densely dotted] (0,1) -- ++(-1.5,0);
\clip (0,-0.1) rectangle (2.8,1.5);
\draw[thin,densely dashed] (-1.5,1) -- ++(1.75,0) to[out=0,in=0,looseness=3] ++(1.5,-1) -- ++(-1.75,0);
\clip (0,-0.1) rectangle (2.8,0.25);
\draw (-1.5,1) -- ++(1.75,0) to[out=0,in=0,looseness=3] ++(1.5,-1) -- ++(-1.75,0);
\end{scope}
\draw (2.7725,0.25) -- ++(0,1.5);
\draw[blue,-latex] (-1.5,1.2) -- node[above,pos=0.8,inner sep=3pt] {\small$\lambda_\beta$} ++(1.5,-1);
\draw[red,-latex] (-1.3,1.3/1.5) -- node[right,pos=0.6,inner sep=2pt] {\small$\mu_\beta$} ++(0,1.5);
\end{scope}

\node at (-0.5,-0.5) {$A - N(\alpha)$};
\node at (1.625,-0.5) {$\cup$};
\node at (4,-0.5) {$T^2 \times [0,1]$};
\node at (6.375,-0.5) {$\cup$};
\node at (8.5,-0.5) {$B - N(\beta)$};
\end{tikzpicture}
\caption{Performing 0-surgery on $\gamma \subset C$ with respect to its $T^2$ framing compresses the torus to a sphere and produces $A_{-1}\#B_{-1}$, because $\gamma$ is isotopic to a peripheral curve of slope $-1$ in each knot exterior.  Doing $+1$-surgery instead produces a splicing in which $\lambda_\alpha$ and $\lambda_\beta$ are glued together, because the orange curve is isotopic to the purple curve and thus bounds a disk.}
\label{fig:splicing-schematic}
\end{figure}
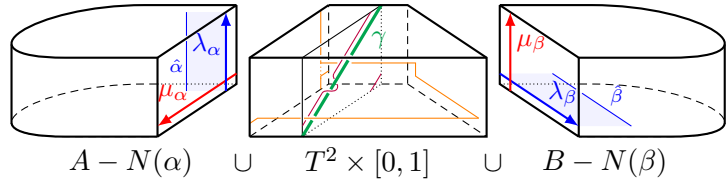

We note that $C$ can  be described, dually, as  $+1$-surgery on the connected sum \[\hat\alpha\#\hat\beta\subset A_{-1}\#B_{-1},\] as   shown on the left in Figure \ref{fig:C}; see  \cite[\S1.1.7]{saveliev} or \cite[Lemma~2.1]{klt}.  This perspective will be helpful in understanding the  commutative diagrams that follow.

\begin{figure}[ht]
\labellist
 \hair 2pt
\pinlabel $\alpha$ at 21 78
\pinlabel {$\beta$} at 134 78
\pinlabel $\alpha$ at 183 78
\pinlabel {$\beta$} at 296 78
\tiny
\pinlabel $S^2$ at 57 9
\pinlabel $A-D^3$ at 39 138
\pinlabel $B-D^3$ at 114 138
\pinlabel $-1$ at 104 112
\pinlabel $-1$ at 43 112
\pinlabel $+1$ at 44 25
\pinlabel $\hat\alpha$ at 231 55
\pinlabel $\hat\beta$ at 244 55
\pinlabel $\gamma$ at 238 102
\endlabellist
\centering
\includegraphics[width=6.7cm]{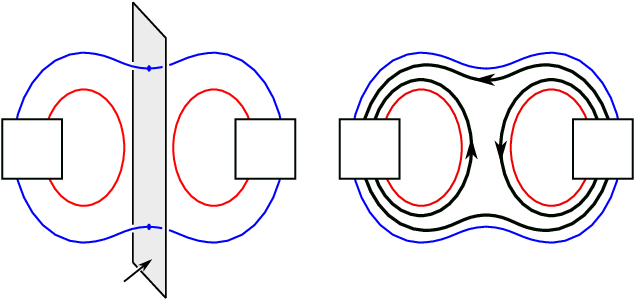}
\caption{Left, a  surgery description of $C$. We will hereafter indicate $-1$, $+1$, and $0$-surgeries using the colors  red, blue, and green, respectively. Right, the curves $\hat\alpha$,  $\hat\beta$, and  $\gamma$ in $C$ shown in black.  }
\label{fig:C}
\end{figure}
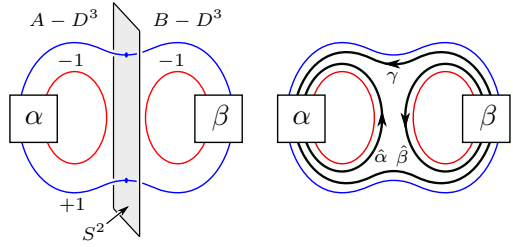

The framed instanton homologies of $C,$ $C_0(\gamma),$ and $C_1(\gamma)$  fit into a surgery exact triangle, a portion of which is shown in the top row of the commutative diagram below:
\begin{equation} \label{eq:diagram1}
\begin{gathered}
\xymatrix{
\Is(C)\ar[r]^-{}\ar[d]&\Is(A_{-1}\#{B_{-1}})\ar[r]^-{e} \ar[d]^-{f}& \Is(S,\lambda_C)\ar[d]^-{d}\\
\Is(A_{0}\#{B},\lambda_A)\ar[r]^-{h}&\Is(A_{0}\#{B_{-1}},\lambda_A)\ar[r]^-{g} & \Is(A_{0}\#{B_{0}},\lambda_A\cup\lambda_B),}
\end{gathered}
\end{equation}
depicted in terms of surgery diagrams in Figure \ref{fig:diag1}.

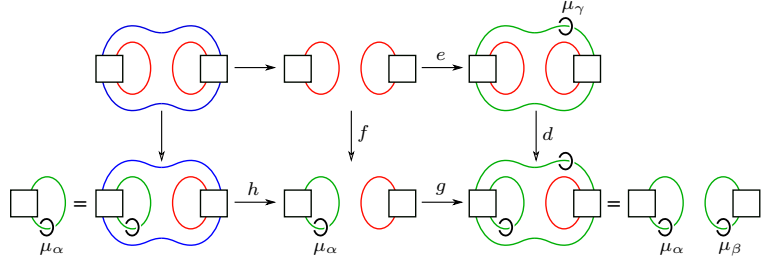
\begin{figure}[ht]
\labellist
\tiny \hair 2pt
\pinlabel $e$ at 272 123
\pinlabel $d$ at 338 75
\pinlabel $g$ at 272 41
\pinlabel $f$ at 224 75
\pinlabel $h$ at 156 41
\pinlabel $\mu_\gamma$ at 354 153
\pinlabel $\mu_\alpha$ at 32 4
\pinlabel $\mu_\alpha$ at 201 4
\pinlabel $\mu_\alpha$ at 415 4
\pinlabel $\mu_\beta$ at 451 4

\endlabellist
\centering
\includegraphics[width=10cm]{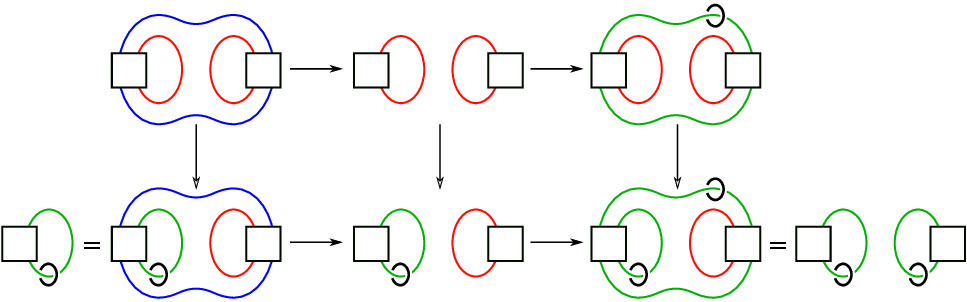}
\caption{Surgery descriptions of the manifolds in the commutative diagram \eqref{eq:diagram1}, and pictures of the meridians $\mu_\alpha$, $\mu_\beta$, and $\mu_\gamma$. We will no longer label the boxes, but will imagine that they contain $\alpha$ and $\beta$ as in Figure \ref{fig:C}. }
\label{fig:diag1}
\end{figure}

The vertical maps in the diagram \eqref{eq:diagram1} are induced by the trace of $0$-surgery on $\hat\alpha$. Here, $\lambda_A, \lambda_B, \lambda_C$ are mod 2 multiples \begin{align*}
\lambda_A &= n_\alpha \mu_\alpha\\
\lambda_B &= n_\beta \mu_\beta\\
\lambda_C &= n_\gamma \mu_\gamma
\end{align*} 
of the meridians $\mu_\alpha,\mu_\beta,\mu_\gamma$ of $\alpha,\beta,\gamma$, respectively, subject to the mod 2 condition  \[n_{\alpha} + n_\beta+n_\gamma=0.\]
Moreover, by the K{\"u}nneth formula for $\Is$ of connected sums, and its naturality with respect to maps induced by split cobordisms \cite[\S7]{scaduto}, if we let $a,b,c$ denote the cobordism maps in the surgery exact triangles
\begin{align*}
\cdots \to \Is(A) \to \Is(A_{-1}) \xrightarrow{a} \Is(A_0, \lambda_A) \to \cdots\hphantom{,}\\
\cdots \to \Is(B) \xrightarrow{c} \Is(B_{-1}) \xrightarrow{b} \Is(B_0, \lambda_B) \to \cdots,
\end{align*}
then we can think of $f,g,h$ as \begin{align*}
f &= a \otimes \id\\
g &= \id\otimes\,b \\
h &= \id\otimes\, c.
\end{align*}
This setup in place, we now prove Theorem \ref{thm:main}, taking the cases \eqref{case1} and \eqref{case2} in turn.

First, assume  we are in case \eqref{case1}, so that  \[\nu^\sharp(\alpha)\leq 0\textrm{ and }\nu^\sharp(\beta)\leq 0.\]Then Lemma \ref{lem:nu} implies that for some choice of $\lambda_A$ and $\lambda_B$, both $a$ and $b$, and therefore $f$ and $g$, are injective. It then follows from the commutativity of the diagram \eqref{eq:diagram1} that $d$ is injective on the image of $e$. On the other hand, \begin{equation}\label{eqn:ker}\Is(S,\lambda_C)_{\Sigma} \subset \ker{d}\end{equation}
since $d$ is induced by $0$-surgery on $\hat\alpha\subset S$, which compresses $\Sigma$, and thus $\Sigma$ is homologous in the associated $2$-handle cobordism to a surface of strictly smaller genus. Therefore, \[\im{e}\cap \Is(S,\lambda_C)_{\Sigma} = 0.\]
Since $\dim \Is(C) \geq \dim \coker(e)$, we have
%which means that $\dim\Is(C)\geq \dim\coker{e}$ satisfies
\[\dim \Is(C) \geq \dim \Is(S,\lambda_C)_{\Sigma}  = 4\cdot\dim\KHI(A,\alpha,g(\alpha))\cdot\dim\KHI(B,\beta,g(\beta))\geq 4,\] by Theorem \ref{thm:nonzero}. Moreover, the leftmost inequality is strict since $\dim \Is(C)$ has the same parity as $ \chi(\Is(C))=1$, proving the theorem in this case.%Since $\chi\HF(C)=1$, we in fact have $\dim \HF(C)\geq 5$.

Next, assume  we are in case \eqref{case2}, so that  \[\nu^\sharp(\alpha) < 0\textrm{ and }\nu^\sharp(\beta) > 0.\] By Lemma \ref{lem:nu}, this implies that $a$ and $c$, and therefore $f$ and $h$, are injective for any choice of $\lambda_A$ and $\lambda_B$. This case will occupy the rest of the proof. Let us henceforth take \[\lambda_A =\lambda_B=\lambda_C=0.\]
The commutative diagram \eqref{eq:diagram1} then becomes: 
\begin{equation} \label{eq:diagram2}
\begin{gathered}
\xymatrix{
 \Is(C)\ar[r]^-{}\ar[d]&\Is(A_{-1}\#{B_{-1}})\ar[r]^-{e} \ar@{^{(}->}[d]^-{f}& \Is(S)\ar[d]^-{d}\\
\Is(A_{0}\#{B})\ar@{^{(}->}[r]^-{h}&\Is(A_{0}\#{B_{-1}})\ar[r]^-{g} & \Is(A_{0}\#{B_{0}}).}
\end{gathered}
\end{equation}
Our goal  will be to once again show that \[\im{e}\cap \Is(S)_{\Sigma} = 0,\] from which the theorem will follow as before. Suppose, for a contradiction, that this intersection contains a nonzero element $e(x)$ for some \[x\in \Is(A_{-1}\#{B_{-1}}).\] Then $f(x)$ is nonzero since $f$ is injective. Note by commutativity and \eqref{eqn:ker} that \[g(f(x))=d(e(x))=0.\] Exactness of the bottom row of \eqref{eq:diagram2} then implies that $f(x) = h(y)$ for some nonzero \[y\in \Is(A_0\# {B}).\]

The maps $f$ and $h$ also fit into the following  diagram in which the rows are portions of surgery exact triangles, and the vertical maps are induced by the trace of $-1$-surgery on $\hat\beta$:
\begin{equation} \label{eq:diagram3}
\begin{gathered}
\xymatrix{
 \Is(A_{-1}\#{B})\ar@{^{(}->}[r]^-{i}\ar@{^{(}->}[d]^-{j}&\Is(A_{0}\# {B})\ar[r]^-{k} \ar@{^{(}->}[d]^-{h}& \Is(A\#B)\ar@{^{(}->}[d]^-{l}\\
\Is(A_{-1}\#{B_{-1}})\ar@{^{(}->}[r]^-{f}&\Is(A_{0}\#{B_{-1}})\ar[r]^-{m} & \Is(A\#{B_{-1}}).}
\end{gathered}
\end{equation}
We claim that $x$ is in the image of $j$. Note that $j$ and $l$ can be viewed as maps of the form $\id \otimes\,c$ and are thus injective.
A similar argument shows that $i$ is injective given that $a$ is. Commutativity and the exactness of the bottom row of this  diagram implies that \[l(k(y))=m(h(y)) = m(f(x)) = 0.\] Then $k(y) = 0$ since $l$ is injective. Exactness of the top row then implies that $y=i(w)$ for some nonzero \[w\in \Is(A_{-1}\# {B}).\] Then \[f(j(w)) = h(i(w)) = h(y) = f(x),\] which implies that $j(w) = x$, as claimed, since $f$ is injective.

The map $j$ from \eqref{eq:diagram3} also fits into the following commutative diagram with $e$ from \eqref{eq:diagram2}:
\begin{equation} \label{eq:diagram4}
\begin{gathered}
\xymatrix{
 &&\Is(A_{0}\#{B_{0}})\ar[d]^{p}\\
 \Is(C_{+1})\ar[r]\ar[d]&\Is(A_{-1}\#{B})\ar[r]^-{n}\ar@{^{(}->}[d]^-{j}&\Is(S_{+1}) \ar[d]^-{o}\\
\Is(C)\ar[r]&\Is(A_{-1}\#{B_{-1}})\ar[r]^-{e}&\Is(S)}
\end{gathered}
\end{equation}
in which the lower vertical maps are induced by the trace of $-1$-surgery on $\hat\beta$; the middle row is a portion of the surgery exact triangle associated to $0$ and $1$-surgery on a copy of $\gamma$ on the splicing torus in $C_{+1}$ with respect to its torus framing; and the rightmost column is part of the surgery  triangle associated to $-1$ and $0$-surgery on $\hat\beta\subset S_{+1}$, as  in Figure \ref{fig:diag2}.

\begin{figure}[ht]
\labellist
\tiny \hair 2pt

\pinlabel $j$ at 171 75
\pinlabel $o$ at 285 75
\pinlabel $p$ at 285 157
\pinlabel $n$ at 206 123
\pinlabel $e$ at 219 39

\endlabellist
\centering
\includegraphics[height=5cm]{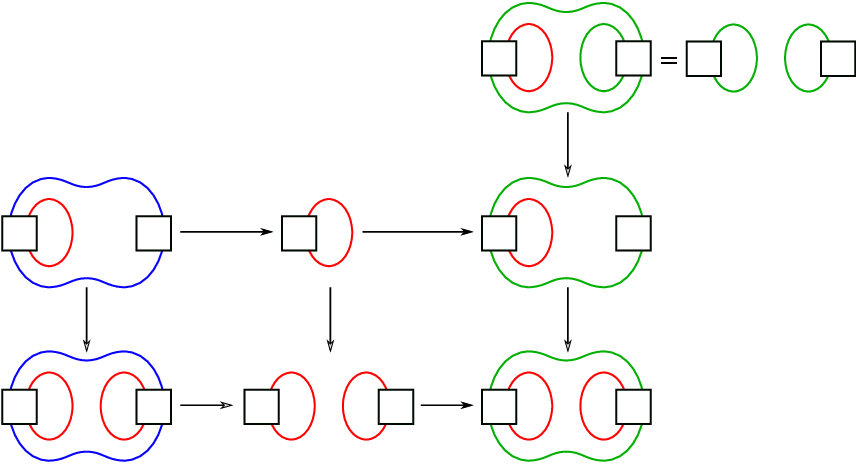}
\caption{Surgery descriptions of the manifolds in the diagram \eqref{eq:diagram4}.}
\label{fig:diag2}
\end{figure}
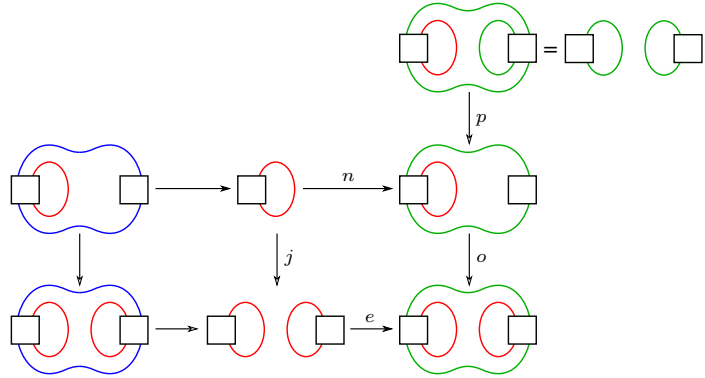

We can write \[n(w) = z + r,\] where \[z\in \Is(S_{+1})_\Sigma \] is a linear combination of elements in the extremal eigenspaces of $\mu(\Sigma)$ acting on $\Is(S_{+1})$, and $r$ is a linear combination of elements in the other eigenspaces. By commutativity,  \[o(z)+o(r) = o(n(w)) = e(j(w)) = e(x) \in \Is(S)_\Sigma.\]
Since the map $o$ respects the eigenspaces of $\mu(\Sigma)$, it follows in exactly the same way as \eqref{eqn:ker}
that $o(r)=0$, which implies by exactness that $r = p(s)$ for some \[s \in\Is(A_{0}\#{B_{0}}).\] 

To complete the proof, we study the element $u(r)$ in the following diagram involving $n$:
\begin{equation} \label{eq:diagram5}
\begin{gathered}
\xymatrix{
&&\Is(A_0\#B_0)\ar[d]^-{p}\\
 \Is(C_{+1})\ar[r]\ar[d]&\Is(A_{-1}\#{B})\ar[r]^-{n}\ar@{^{(}->}[d]^-{q}&\Is(S_{+1}) \ar[d]^-{u}\\
\Is(A_0\#{B_{+1}},\mu_\alpha)\ar[r]&\Is(A_0\#{B},\mu_\alpha)\ar@{^{(}->}[r]^-{v}&\Is(A_0\#B_0,\mu_\alpha\cup\mu_\beta),}
\end{gathered}
\end{equation}
in which the rows are portions of surgery exact triangles corresponding to $0$- and $1$-surgery on copies of $\gamma$. In the manifolds on the bottom row, this copy of $\gamma$ is isotopic to $\beta$, so this row can be viewed as  the  exact triangle corresponding to $0$- and $1$-surgery on $\beta$, as indicated in Figure \ref{fig:diag3}. The lower vertical maps are induced by the trace of $0$-surgery on $\hat\alpha$.

\begin{figure}[ht]
\labellist
\tiny \hair 2pt
\pinlabel $u$ at 339 83
\pinlabel $p$ at 339 166
\pinlabel $n$ at 273 132
\pinlabel $v$ at 273 49
\pinlabel $q$ at 250 83
\pinlabel $\mu_{\alpha\beta}$ at 321 -2
\pinlabel $\mu_{\alpha}$ at 247 13
\pinlabel $\mu_{\alpha}$ at 416 13
\pinlabel $\mu_{\alpha}$ at 33 13
\pinlabel $\mu_{\beta}$ at 451 13
\endlabellist
\centering
\includegraphics[height=5.2cm]{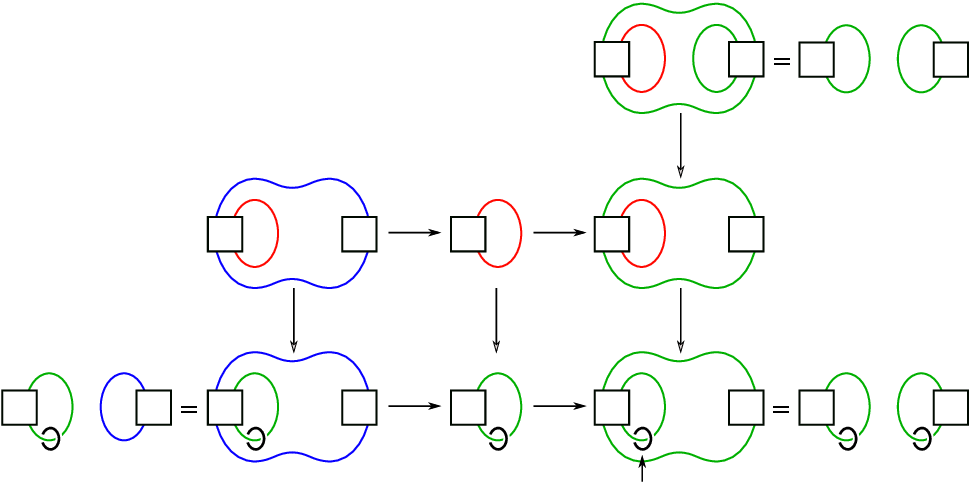}
\caption{Surgery descriptions of the manifolds in the diagram \eqref{eq:diagram5}.}
\label{fig:diag3}
\end{figure}

The map $q$ is of the form $a'\otimes \id$, where $a'$ is the map in the surgery exact triangle \[\cdots \to \Is(A) \to \Is(A_{-1}) \xrightarrow{a'} \Is(A_0, \mu_\alpha) \to \cdots.\] Since $\nu^\sharp(\alpha)<0$, Lemma \ref{lem:nu} implies that $a'$, and therefore $q$, is injective. Similarly, we can view $v$ as a map of the form $\id\otimes\,t$, where $t$ is the map in the surgery exact triangle
\[
\cdots \to \Is(B) \xrightarrow{t} \Is(B_{0},\mu_\beta) \xrightarrow{} \Is(B_1) \to \cdots
\]
Since $\nu^\sharp(\beta)>0$,  Lemma \ref{lem:nu} implies that $t$, and hence $v$, is injective. It follows that $v(q(w))$ is nonzero. By commutativity of \eqref{eq:diagram5}, \[v(q(w)) = u(n(w)) = u(z) + u(r).\] Note that $u(z)=0$ by the same argument as in \eqref{eqn:ker}: namely, that $u$ respects the eigenspaces of $\mu(\Sigma)$, we have \[z\in \Is(S_{+1})_\Sigma,\] and $0$-surgery on $\hat\alpha\subset S_{+1}$ compresses $\Sigma$. Thus, $u(r)$ is nonzero. On the other hand, \[u(r) = u(p(s)),\] and the composition \begin{equation}\label{eqn:sphere} \Is(A_{0}\#{B_{0}})\xrightarrow{p}\Is(S_{+1})\xrightarrow{u} \Is(A_{0}\#{B_{0}},\mu_\alpha\cup\mu_\beta)\end{equation} is identically zero since, as we will show momentarily, the composite cobordism contains a 2-sphere $V$ with  self-intersection $0$ on which the restriction of the $\mathit{SO}(3)$-bundle is nontrivial \cite[Proof of Proposition 3.3]{bs-concordance}, a contradiction.

The sphere $V$ is depicted in Figure~\ref{fig:sphere}, as the union of an annulus $V_\gamma\subset S_{+1}$ with disks $V_\alpha$ and $V_\beta$, where $V_\alpha$ is the core of the 2-handle attachment that induces the map $u$, and $V_\beta$ is the co-core of the 2-handle attachment that induces $p$. The mod 2 class \[[\mu_{\alpha}\cup\mu_\beta]\in H_1(A_{0}\#B_{0};\Z/2\Z)\] is represented by the meridional curve $\mu_{\alpha\beta}$ shown in Figure \ref{fig:diag3}. This curve bounds a properly embedded disk $D$ in the cobordism inducing the map $u$, obtained by pushing the interior of the meridional disk bounded by $\mu_{\alpha\beta}$ in the surgery diagram into the interior of the cobordism. The second Stiefel-Whitney class of the $\mathit{SO}(3)$-bundle on the composite cobordism is Poincar{\'e} dual to $D$; see \cite[\S2.3]{abds}.
The evaluation of this class on $V$ is therefore given by the mod 2 intersection number of $V$ with $D$. But this is equal to the mod 2 intersection number of $V$ with $\mu_{\alpha\beta}$, as indicated in Figure \ref{fig:sphere},  which is 1.
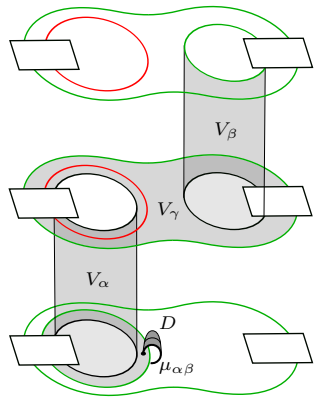
\begin{figure}[ht]
\labellist
\tiny \hair 2pt
\pinlabel $V_\gamma$ at 54.5 62
%\pinlabel $*$ at 55.5 72.5
\pinlabel $V_\beta$ at 72 86
\pinlabel $V_\alpha$ at 33 41
\pinlabel $\mu_{\alpha\beta}$ at 57 14
\pinlabel $D$ at 54 28
\endlabellist
\centering
\includegraphics[height=5.3cm]{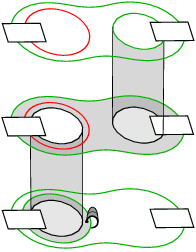}
\caption{The 2-sphere with trivial self-intersection in the cobordism corresponding to the composition in \eqref{eqn:sphere}. It is formed as the union of an annulus $V_\gamma$ with disks $V_\alpha$ and $V_\beta$. It intersects $\mu_{\alpha\beta}$, and thus $D$, in one point.} %$*$.}
\label{fig:sphere}
\end{figure}
\end{proof}

\begin{proof}[Proof of Corollary \ref{cor:toroidal}] If $Y$ is not a homology sphere, then there is a nontrivial homomorphism from $\pi_1(Y)$ to $\SU(2)$ factoring through $H_1(Y)$. We  may thus assume that $Y$ is a homology sphere, in which case it  can be expressed as the splice of  nontrivial knots $K_1\subset Y_1$ and $K_2\subset Y_2$ in homology spheres, in which the complements are glued by a map that identifies meridians with Seifert longitudes. There is a $\pi_1$-surjective map from $Y$ to each $Y_i$, so if some $\pi_1(Y_i)$ admits a nontrivial $\SU(2)$-representation then so does $\pi_1(Y)$. We may thus assume that neither $\pi_1(Y_1)$ nor $\pi_1(Y_2)$ admits a nontrivial $\SU(2)$-representation. Then these $Y_i$ are instanton $L$-spaces, as explained in the introduction, 
so Theorem \ref{thm:main} tells us that $Y$ is not an instanton $L$-space. But this in turn implies that $\pi_1(Y)$ admits a nontrivial homomorphism to $\SU(2)$.
\end{proof}

\begin{proof}[Proof of Corollary \ref{cor:SL2}]
As in the previous proof, we may assume $Y$ is a  homology sphere.  We can further assume $Y$ is prime, in which case $Y$ is hyperbolic, Seifert fibered, or toroidal, by geometrization. If $Y$ is hyperbolic then the holonomy representation $\pi_1(Y)\to\mathit{PSL}(2,\C)$ lifts to a nontrivial representation $\pi_1(Y)\to\SL(2,\C)$ \cite[Proposition~3.1.1]{cs-splittings}.  If $Y\not\cong S^3$ is Seifert fibered then $\pi_1(Y)$ admits a nontrivial homomorphism to $\SU(2)$, and hence to $\SL(2,\C)$ \cite{fs-seifert}; see also \cite[Theorem 5.1]{bs-stein}. Finally, if $Y$ is toroidal then $\pi_1(Y)$ admits a nontrivial homomorphism to $\SU(2)$, and hence   $\SL(2,\C)$, by Corollary \ref{cor:toroidal}.
\end{proof}

\section{The framed instanton homology of a longitudinal splice}\label{sec:appendix}

The goal of this section is to prove Theorem \ref{thm:nonzero}. We will assume familiarity with sutured instanton homology at the level of \cite{km-excision}. All sutured manifolds in this section are assumed to be balanced with nonempty sutures.

Given a sutured manifold $M = (M,\gamma)$ and a positive integer $k$, let $M(k)$ be the sutured manifold obtained by removing $k$ 3-balls from the interior of $M$ and adding a circle on each of the resulting 2-sphere boundary components to the suture $\gamma$. Recall that  \[\Is(Y) \cong \SHI(Y(1))\] for any closed 3-manifold $Y$.  We start with the preliminary lemma below.

\begin{lemma} 
\label{lem:connect}For any two sutured manifolds $M$ and $N$, \[\dim \SHI(M\#N) = \dim\SHI(M) \cdot \dim \SHI(N(1)).\]
\end{lemma}

\begin{proof}
Note that $M\#N$ is obtained from $M\sqcup N(1) \cong M \sqcup S^3(1)\#N$
by attaching a contact 1-handle. Since such attachments preserve sutured instanton homology \cite[\S3.2]{bs-instanton}, we have \[\SHI(M\#N) \cong \SHI(M\sqcup N(1))\cong \SHI(M)\otimes \SHI(N(1)),\] proving the lemma.
\end{proof}

\begin{corollary} \label{cor:hole}For any sutured manifold $M$ and positive integer $k$, \[\dim\SHI(M(k)) = 2^k\cdot \dim\SHI(M).\]
\end{corollary}

\begin{proof}
Note that $(S^1\times S^2)(1)$ can be obtained   by attaching a contact 1-handle to $S^3(2)$, so  \[\dim\SHI(S^3(2)) = \dim\SHI((S^1\times S^2)(1)) = \dim \Is(S^1\times S^2) = 2.\] Since $M(k)\cong M(k-1)\# S^3(1)$,  Lemma \ref{lem:connect} says that \[\dim\SHI(M(k)) = \dim\SHI(M(k-1))\cdot \dim\SHI(S^3(2)) = 2\cdot \dim\SHI(M(k-1)), \]  and the corollary follows by induction.
\end{proof}

\begin{corollary}\label{cor:nonzero} For any nontrivial knot $\alpha$ in a closed, oriented homology sphere $A$, \[\dim\KHI(A,\alpha,g(\alpha)) \geq 1.\]
\end{corollary}

\begin{proof} Let $\Sigma$ be a minimal genus Seifert surface for $\alpha$, and let $A(\Sigma)$ be the sutured manifold obtained from $A$ by removing a neighborhood of $\alpha$ and then of $\Sigma$, with suture consisting of one curve  isotopic to $\partial\Sigma$. Then  \[\KHI(A,\alpha,g(\alpha))\cong \SHI(A(\Sigma)).\] %By Corollary \ref{cor:nonzero}, it therefore suffices to show that $\SHI(A(\Sigma)(1))$ is nontrivial.
Sutured instanton homology is nonzero for any taut sutured manifold \cite[Theorem 7.2]{km-excision}. Since $\Sigma$ has minimal genus, the only way $A(\Sigma)$  can fail to be taut is if it is not prime. In this case, we can write \[A(\Sigma) = A(\Sigma)' \# Y,\] where $A(\Sigma)'$ is a taut sutured manifold and $Y$ is a  homology sphere. Then \[A(\Sigma)(1) = A(\Sigma)' \# Y(1),\] whence Lemma \ref{lem:connect} and Corollary \ref{cor:hole} tell us that \begin{align*}\dim\SHI(A(\Sigma)) &=(1/2)\cdot\dim\SHI(A(\Sigma)(1))\\
&= (1/2)\cdot \dim\SHI(A(\Sigma)')\cdot \dim \SHI(Y(2))\\
&= \dim\SHI(A(\Sigma)')\cdot \dim \SHI(Y(1)) = \dim\SHI(A(\Sigma)')\cdot \dim \Is(Y), 
\end{align*} which is nonzero since $A(\Sigma)'$ is taut and $\chi(\Is(Y))=1$.
\end{proof}

We will require one more lemma before proving Theorem \ref{thm:nonzero}.
\begin{lemma}
\label{lem:cut}
Let  $Y$ be a closed, oriented $3$-manifold. Let  $\Sigma$ be a closed, oriented surface in $Y$ of genus at least two, and $\lambda$ a multicurve in $Y$ which intersects some surface of positive genus disjoint from $\Sigma$ in an odd number of points. Suppose $(Y',\lambda')$ is obtained by cutting $(Y,\lambda)$ open along $\Sigma$ and regluing by a homeomorphism of $\Sigma$ which preserves $\Sigma\cap \lambda$. Then \[I_*(Y|\Sigma)_\lambda\cong I_*(Y'|\Sigma)_{\lambda'}. \] \end{lemma}

\begin{proof} When $\Sigma\cap \lambda$ consists of an odd number of points, this is simply excision. In general, it suffices to consider the case in which the regluing homeomorphism is a positive Dehn twist along an essential curve $\gamma\subset \Sigma$. In this case, $Y'$ is the result of $-1$-surgery on $\gamma\subset Y$. Consider the  exact triangle relating the Floer homologies of $Y$, $Y_{-1}(\gamma)\cong Y'$, and $Y_0(\gamma)$:
\[\cdots \to I_*(Y)_{\lambda} \xrightarrow{f} I_*(Y')_{\lambda'} \xrightarrow{} I_*(Y_0(\lambda))_{\lambda''\cup \mu} \to \cdots,\]
where $\mu$ is a meridian of $\gamma$. The maps in this triangle respect the eigenspaces of the operators $\mu(\Sigma)$ and $\mu(\pt)$. On the other hand, $0$-surgery on $\gamma$ compresses $\Sigma$, which implies that \[I_*(Y_0(\lambda)|\Sigma)_{\lambda''\cup \mu}=0.\] It follows that $f$ restricts to an isomorphism from $I_*(Y|\Sigma)_\lambda$ to $ I_*(Y'|\Sigma)_{\lambda'}$.
\end{proof}

This immediately implies the corollary below, since $\Is(Y,\lambda|\Sigma)=I_*(Y\#T^3|\Sigma)_{\lambda\cup \lambda_T}$.
\begin{corollary}\label{cor:cut}
Let  $Y$ be a closed, oriented $3$-manifold. Let  $\Sigma$ be a closed, oriented surface in $Y$ of genus at least two, and $\lambda$  a multicurve in $Y$. Suppose $(Y',\lambda')$ is obtained by cutting $(Y,\lambda)$ open along $\Sigma$ and regluing by a homeomorphism of $\Sigma$ which preserves $\Sigma\cap \lambda$. Then \[\Is(Y,\lambda|\Sigma)\cong \Is(Y',\lambda'|\Sigma).  \] 
\end{corollary}

\begin{proof}[Proof of Theorem \ref{thm:nonzero}] Let $\alpha\subset A$ and  $\beta\subset B$ be as in hypothesis of the theorem. Let $S$ be any splice formed by gluing the   complements \[S = (A- N(\alpha)) \cup ({B- N(\beta)})\] via an orientation-reversing homeomorphism of their boundaries which identifies the Seifert longitudes of these knots. Let $\Sigma$ be the closed surface in $S$ of genus at least two obtained as the union of minimal genus Seifert surfaces $\Sigma_\alpha \subset A-N(\alpha)$ and $\Sigma_\beta \subset B-N(\beta)$ for $\alpha$ and $\beta$. %Note that we can, by cutting  $S$ open along $\Sigma$ and regluing by an appropriate composition of Dehn twists along $\partial\Sigma_\alpha\subset \Sigma$, obtain 
%the splice in which the meridians of the knots are identified in addition to their Seifert longitudes. Since this cutting and regluing does not change $\Is(S,\lambda|\Sigma)$, per Corollary \ref{cor:cut}, we are free to  assume that $S$ \emph{is} the splice in which both meridians and Seifert longitudes are identified.
Since $A$ and $B$ are homology spheres, \[H_1(S;\Z/2\Z)\cong \Z/2\Z,\] generated by the meridian $\mu$ of $\alpha$, say. Then $\lambda = 0$ or $\mu$. We will first show that \begin{equation}\label{eqn:0mu}\Is(S,0|\Sigma)\cong \Is(S,\mu|\Sigma).\end{equation} Then we will show by an excision argument that \begin{equation}\label{eqn:isosplice}\dim\Is(S,\mu|\Sigma) = 2\cdot\dim\KHI(A,\alpha,g(\alpha))\cdot \dim\KHI(B,\beta,g(\beta)).\end{equation} Theorem \ref{thm:nonzero} will then follow immediately from \eqref{eqn:isosplice} together with  \eqref{eqn:symm} and Corollary \ref{cor:nonzero}.

Let $A(\Sigma_\alpha)$ and $B(\Sigma_\beta)$ be the standard  sutured  complements of the Seifert surfaces $\Sigma_\alpha\subset A$ and $\Sigma_\beta\subset B$. Each has vertical boundary  an annular neighborhood of the suture,  \[\partial_v A(\Sigma_\alpha) \cong \partial_v B(\Sigma_\beta) \cong (\partial \Sigma_\alpha = \partial \Sigma_\beta)\times [0,1],\] while their horizontal boundaries are given by
 \begin{align*}
\partial_h A(\Sigma_\alpha) &= \Sigma_\alpha\sqcup\Sigma_\alpha\\
\partial_h B(\Sigma_\beta) &= \Sigma_\beta\sqcup\Sigma_\beta.
\end{align*} 
The manifold obtained by cutting $S$ open along $\Sigma$ can alternatively be viewed as the result \[\bar S = A(\Sigma_\alpha)\cup B(\Sigma_\beta)\]  of  gluing $A(\Sigma_\alpha)$ and $B(\Sigma_\beta)$  along their vertical boundaries. We can view \[\partial \bar S= \Sigma_{1}\sqcup \Sigma_{0},\] as the disjoint union of  two identical copies of $\Sigma$, so that $S$ is obtained from $\bar S$ by gluing $\Sigma_{1}$ to $\Sigma_{0}$ by the identity. 

There is another convenient way to think about $\bar S$ and $S$.
Consider the manifolds \begin{align*}
A' &= A(\Sigma_\alpha) \cup (\Sigma_\beta\times[1/2,1])\\
B' &= B(\Sigma_\beta) \cup (\Sigma_\alpha\times[0,1/2]),
\end{align*} obtained by gluing their constituent pieces by homeomorphisms 
\begin{align*}
\partial\Sigma_\beta\times[1/2,1]&\to \partial_vA(\Sigma_\alpha)\\
\partial\Sigma_\alpha\times[0,1/2]&\to \partial_vB(\Sigma_\beta).
\end{align*} The boundaries of $A'$ and $B'$ can each be viewed as  consisting of two identical copies of $\Sigma$, \begin{align*}
\partial A'&=\Sigma_1 \sqcup \Sigma_{1/2}\\
\partial B'&= \Sigma_0 \sqcup \Sigma_{1/2},
\end{align*}
so that $\bar S$ is obtained by gluing $A'$ to $B'$ according to the identity  on $\Sigma_{1/2}$, and $S$ is obtained by additionally gluing $\Sigma_1$ to $\Sigma_0$ by the identity, as indicated in Figure \ref{fig:splitting}.

\begin{figure}[ht]
\labellist
\tiny \hair 2pt
\pinlabel $A(\Sigma_{\alpha})$ at 37 48
\pinlabel $B(\Sigma_{\beta})$ at 78 48
\pinlabel $A(\Sigma_{\alpha})$ at 170 57
\pinlabel $\Sigma_\beta\times I$ at 210 57
\pinlabel $B(\Sigma_\beta)$ at 210 37
\pinlabel $\Sigma_\alpha\times I$ at 170 37
\pinlabel $\Sigma_{1/2}$ at 244 47
\pinlabel $\Sigma_1$ at 239 68
\pinlabel $\Sigma_0$ at 239 27.5

\pinlabel $\Sigma_1$ at 107.5 68
\pinlabel $\Sigma_0$ at 107.5 27.5

\pinlabel $A(\Sigma_{\alpha})$ at 301 70
\pinlabel $\Sigma_\beta\times I$ at 342 70
\pinlabel $B(\Sigma_\beta)$ at 342 25
\pinlabel $\Sigma_\alpha\times I$ at 301 25

\pinlabel $\bar A$ at 322 92
\pinlabel $\bar B$ at 322 4

\pinlabel $S$ at 58 16
\pinlabel $S$ at 190 16

\pinlabel $\Sigma_{1/2}$ at 375 58
\pinlabel $\Sigma_{1/2}$ at 375 34.5
\pinlabel $\Sigma_1$ at 370 79.5
\pinlabel $\Sigma_0$ at 370 14.5

\pinlabel $\cong$ at 121 47

\endlabellist
\centering
\includegraphics[height=2.6cm]{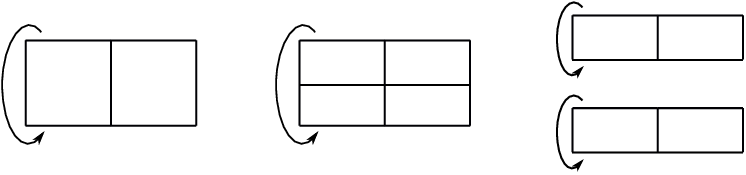}
\caption{Left, $S$ obtained from $A(\Sigma_\alpha)\cup B(\Sigma_\beta)$ via the identification $\Sigma_1\sim \Sigma_0$. Middle, $S$ obtained from $A'\cup B'$ by identifying $\Sigma_1\sim\Sigma_0$. Right, $\bar A\sqcup \bar B$ obtained from $A'\sqcup B'$ via identifications $\Sigma_1\sim\Sigma_{1/2}$ and $\Sigma_{1/2}\sim\Sigma_0$.} %$*$.}
\label{fig:splitting}
\end{figure}
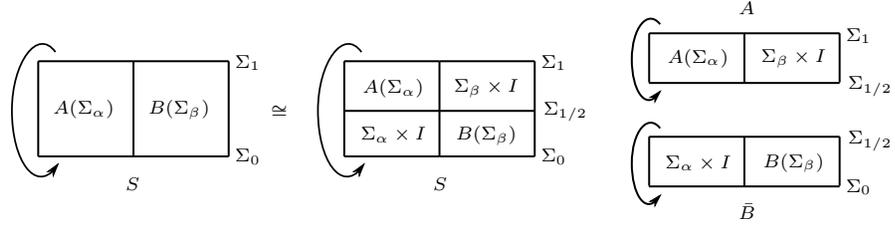

Pick essential curves $c_\alpha\subset \Sigma_\alpha$ and $c_\beta\subset \Sigma_\beta$, and let $\delta$ be a curve in $\Sigma_{\alpha}$ which intersects $c_\alpha$ transversely in one point. Since $A$ is a homology sphere and $\delta$ has linking number zero with $\alpha$, the curve $\delta$ is homologically trivial in $A - N(\alpha)$ and hence in $S$,
so \begin{equation}\label{eqn:iso1}\Is(S,0|\Sigma)\cong \Is(S,\delta|\Sigma).\end{equation} Let $S'$ be the  manifold obtained by gluing $A'$ to $B'$ by a homeomorphism $\partial A' \to \partial B'$ %given by 
%\[\Sigma_1\xrightarrow{}\Sigma_0 \textrm{ and }\Sigma_{1/2}\xrightarrow{}\Sigma_{1/2}\] 
sending \begin{align*}c_\beta\times\{1\}&\mapsto c_\alpha\times\{0\}\\
c_\beta\times\{1/2\} &\mapsto c_\alpha\times\{1/2\},\end{align*}  so that the cylinders \begin{align*}
c_\beta\times[1/2,1]\subset A' \\
c_\alpha\times[0,1/2]\subset B'\end{align*} glue  to form a torus  $T\subset S'$ which intersects the curve $\delta$ in a single point. Note that  $S'$ can be obtained by cutting $S$  along $\Sigma_0\sim\Sigma_1$ and $\Sigma_{1/2}$ and regluing. Therefore, \begin{equation}\label{eqn:iso2}\Is(S,\delta|\Sigma)\cong \Is(S',\delta|\Sigma),\end{equation}  by Corollary \ref{cor:cut}.

Now consider the 3-torus
$T^3 = T^2\times S^1$ containing the curve \[\lambda = \pt\times S^1\subset T^3.\] Let $c\subset T^2$ be an essential curve, and $d\subset T^2\times\pt$ a curve intersecting $c\times\pt$ transversely in one point. There is a cobordism
\begin{equation}\label{eqn:excisioncob}(S',\delta)\sqcup (T^3, d\cup \lambda)\to (S',\delta\cup\lambda)\end{equation}
associated to cutting $(S',\delta)$ and $(T^3, d\cup \lambda)$ open along the tori \begin{align*}T&\subset S'\\
c\times S^1&\subset T^3
\end{align*} and regluing by a homeomorphism which connects these pieces, as illustrated in Figure \ref{fig:excision}. By excision \cite[Theorem 7.7]{km-excision} and the computation $I_*(T^3|T^2)_{d\cup\lambda} \cong \Q$ of \cite[Proposition~7.8]{km-excision},
this cobordism induces an isomorphism
\begin{equation}\label{eqn:iso3}
\Is(S',\delta|\Sigma) \cong \Is(S',\delta|\Sigma)\otimes I_*(T^3|T^2)_{d\cup\lambda}\to \Is(S',\delta\cup \lambda|\Sigma).\end{equation}
Since $\lambda$ intersects $\Sigma$ in one point, another application of Corollary \ref{cor:cut} tells us that
\begin{equation}\label{eqn:iso4}\Is(S',\delta\cup \lambda|\Sigma)\cong \Is(S,\delta\cup \lambda|\Sigma)\cong\Is(S,\delta\cup \mu|\Sigma),\end{equation} where the second isomorphism comes from the fact that $\lambda$ and $\mu$ both represent the generator of $H_1(S;\Z/2\Z)\cong \Z/2\Z$.
Finally, \begin{equation}
\label{eqn:iso5}
\Is(S,\delta\cup \mu|\Sigma)\cong \Is(S, \mu|\Sigma),
\end{equation}  since $\delta$ is nullhomologous in $S$. Combining  \eqref{eqn:iso1}-\eqref{eqn:iso5} proves the isomorphism \eqref{eqn:0mu}.

\begin{figure}[ht]
\labellist
\tiny \hair 2pt
\pinlabel $S'$ at 57 9
\pinlabel $T^3$ at 185 9
\pinlabel $S'$ at 350 9
\pinlabel $c_\alpha$ at 95 97
\pinlabel $c$ at 148 97
\pinlabel $\delta$ at 59 115
\pinlabel $d$ at 183 115
\pinlabel $\lambda$ at 223 88
\pinlabel $\lambda$ at 390 88
\pinlabel $\delta$ at 350 115

\endlabellist
\centering
\includegraphics[height=2.5cm]{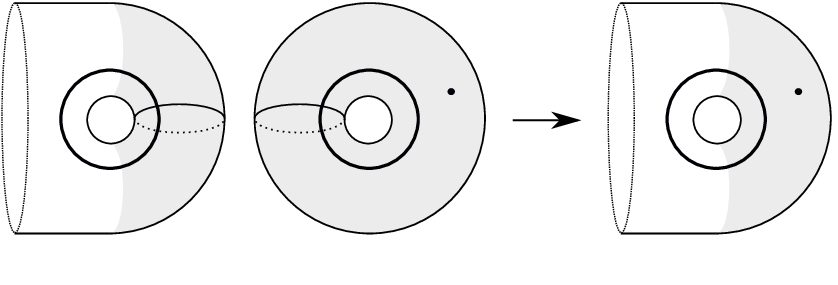}
\caption{A schematic of the cobordism \eqref{eqn:excisioncob}. The cutting and regluing occurs in the product of the gray region with $S^1$. The torus $T\subset S'$ is swept out by $c_\alpha$, and $\lambda$ is swept out by the points indicated on these surfaces.} %$*$.}
\label{fig:excision}
\end{figure}
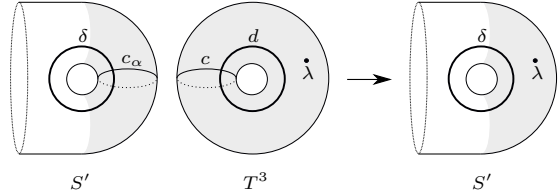

To prove \eqref{eqn:isosplice}, we form $S$ by gluing $A'$ to $B'$ as described above.  There is a cobordism \[(S,\mu)\to (\bar A, \mu_A)\sqcup (\bar B, \mu_B),\]
associated to cutting $S$ open along $\Sigma_{1/2}$ and $\Sigma_1\sim\Sigma_0$ and regluing, where\begin{align*} \bar A& = A'/(\Sigma_1\sim\Sigma_{1/2})\\
\bar B& = B'/(\Sigma_0\sim\Sigma_{1/2})
\end{align*} are the closed 3-manifolds obtained from each of $A'$ and $B'$ by gluing their boundary components together by the identity, as  in Figure \ref{fig:splitting}, and $\mu_\alpha\subset \bar A$ and $\mu_\beta\subset \bar B$  intersect  $\Sigma\subset \bar A$ and $\Sigma\subset \bar B$ in one point, respectively. By excision, this cobordism induces an isomorphism \[\Is(S,\mu|\Sigma) = I_*(S\# T^3|\Sigma)_{\mu\cup\lambda_T}\to I_*(\bar A\#T^3|\Sigma)_{\mu_\alpha\cup \lambda_T} \otimes I_*(\bar B|\Sigma)_{\mu_\beta},\] but the tensor product on the right is, by definition \cite[\S7]{km-excision}, precisely \[\SHI(A(\Sigma_\alpha)(1))\otimes\SHI(B(\Sigma_\beta)),\] which, by Corollary \ref{cor:hole}, has dimension
\begin{multline*}
\quad 2\cdot \dim \SHI(A(\Sigma_\alpha))\cdot \dim\SHI(B(\Sigma_\beta))\\
= 2\cdot\dim\KHI(A,\alpha,g(\alpha))\cdot\dim \KHI(B,\beta,g(\beta)), \quad
\end{multline*}
as desired.
\end{proof}

\begin{remark}
One can prove the Heegaard Floer homology analogue of Theorem \ref{thm:nonzero} by a couple of different arguments. One can, for instance, simply mimic the argument above, using Lekili's  excision result for Heegaard Floer homology in \cite{lekili-thesis}. 
\end{remark}

\bibliographystyle{alpha}
\bibliography{References}

\end{document}